\documentclass[reqno]{amsart}
\usepackage{amsmath}
\usepackage{amsfonts}

\newtheorem{theorem}{Theorem}
\theoremstyle{plain}
\newtheorem{acknowledgement}{Acknowledgement}

\newtheorem{definition}{Definition}

\newtheorem{proposition}{Proposition}
\newtheorem{remark}{Remark}

\numberwithin{equation}{section}
\input{tcilatex}

\begin{document}
\title[Homogenization of a nonlocal diffusion equation]{Homogenization of a
Wilson-Cowan model for neural fields}
\author{Nils Svanstedt}
\address{N. Svanstedt, Department of Mathematical Sciences, G\"{o}teborg
University, SE-412 96 G\"{o}teborg, Sweden}
\email{nilss@math.chalmers.se}
\author{Jean Louis Woukeng}
\address{Jean Louis Woukeng, Department of Mathematics and Computer Science,
University of Dschang, P.O. Box 67, Dschang, Cameroon}
\email{jwoukeng@yahoo.fr}
\date{}

\begin{abstract}
Homogenization of Wilson-Cowan type of nonlocal neural field models is
investigated. Motivated by the presence of a convolution terms in this type
of models, we first prove some general convergence results related to
convolution sequences. We then apply these results to the homogenization
problem of the Wilson-Cowan type model in a general deterministic setting.
Key ingredients in this study are the notion of algebras with mean value and
the related concept of sigma-convergence.
\end{abstract}

\maketitle

\section{Introduction}

Experiments/observations through different EEG, fMRI, MEG and optical
imaging techniques reveal electrical activity patterns spanning over several
centimeter of brain tissue i.e. of length scales much larger than the
spatial extent of one single neuron. Moreover, these structures have a
lifetime which is much larger than the lifetime of typical firing time for a
neuron. Now, as the cortex obviously is a heterogeneous medium possessing
many different spatial and temporal scales it is a need to have rigorous
ways of determining how the spatio-temporal microstructure is stored in mean
field models for the brain activity. One way of doing this is by means of 
\emph{homogenization theory} based on multi-scale convergence techniques.
The problem of homogenization or \emph{scaling} is to determine from data or
local characteristics, the effective models representing the macroscopic
behavior of mesoscopically inhomogeneous media.

In the quest of studying and understanding the neurodynamics, several
heuristic models have been developed. The classical \emph{leaky integrator
unit} model \cite{HKP91, SLM74, WC72} described by ordinary differential
equations has given rise to the most studied \emph{Wilson-Cowan type of
models} \cite{WC73} (see also Amari \cite{A}) which, in the one dimensional
space, reads as

\begin{equation}
\frac{\partial }{\partial t}u(x,t)=-u(x,t)+\int_{-\infty }^{\infty
}J(x^{\prime },x)f(u(x^{\prime },t))dx^{\prime }.  \label{1.1}
\end{equation}%
Here $u(x,t)$ denotes the neural field which measures the local activity of
neurons at position $x\in {\mathbb{R}}$. The integral represents the
synaptic input where the function $J(x,x^{\prime })$ measures the strength
of connections between neurons at positions $x$ and $x^{\prime }$. We refer
to this function as the connectivity function. The function $f$ is the
firing rate function. Equation (\ref{1.1}) models the neural field in a
homogeneous medium. For more details concerning this model we refer to e.g.
Wilson and Cowan \cite{WC73} and Amari \cite{A} and the more recent work by
Coombes \cite{Coombes2005} and the references therein.

A drawback with the model (\ref{1.1}) above is that does not take into
account that the brain is very heterogeneous with a structure that exhibits
multiple spatial scales ranging from micro- to decimeter. In addition the
dynamical activities are taking place on multiple time scales. In order to
capture several such properties we are led to allow the connectivity to
depend on both space and time and also to depend on multiple spatial and
temporal scales. In the present work we will consider time independent
connectivity kernels.

An obvious way to impose heterogeneity in the well-accepted Wilson-Cowan or
Amari model (\ref{1.1}) is to replace the connectivity function $J$ by a
function $J_{\varepsilon }$, where the parameter $\varepsilon >0$ measures
the heterogeneity of the brain tissue.

The homogenization procedure has been used in many applied science fields to
upscale various mathematical models. As far as the neural field models are
concerned, there are very few works dealing with homogenization techniques.
See e.g., \cite{KFB} or \cite{CLSSW2011}. In \cite{KFB} a homogenization
based approach for studying nonlocal heterogeneous neural field models of
Wilson-Cowan type based on multi-scale expansion techniques is initiated. A
heterogeneous neural field model is also advocated in \cite{CLSSW2011} where
the point of departure is the parametrized Wilson-Cowan model 
\begin{equation}
\frac{\partial }{\partial t}u_{\varepsilon }(x,t)=-u_{\varepsilon
}(x,t)+\int_{\mathbb{R}^{N}}J_{\varepsilon }(x-x^{\prime })f(u_{\varepsilon
}(x^{\prime },t))dx^{\prime },\quad x\in \mathbb{R}^{N},\quad t>0
\label{1.2}
\end{equation}%
where $u_{\varepsilon }$ denotes the electrical activity level field, $f$
the firing rate function and $J_{\varepsilon }=J_{\varepsilon
}(x)=J(x,x/\varepsilon )$ the connectivity kernel which by assumption is
periodic in the second argument $y=x/\varepsilon $. In the present work,
under a general deterministic assumption on the kernel (including the
periodicity assumption and the almost periodicity assumption) we prove
rigorously (see Theorem \ref{t4.1}) that, as $\varepsilon \rightarrow 0$,
the solution $u_{\varepsilon }$ to the Wilson-Cowan model (\ref{1.2})
converges to the solution $u_{0}$ of a homogenized Wilson-Cowan equation 
\begin{equation}
\frac{\partial }{\partial t}u_{0}(x,t,y)=-u_{0}(x,t,y)+(J\ast \ast
f(u_{0}))(x,t,y).  \label{1.3}
\end{equation}%
In the special periodic case, Equation (\ref{1.3}) reads as 
\begin{equation*}
\frac{\partial }{\partial t}u_{0}(x,t,y)=-u_{0}(x,t,y)+\int_{\mathbb{R}%
^{N}}\int_{Y}J(x-x^{\prime },y-y^{\prime })f(u_{0}(x^{\prime },t,y^{\prime
}))dy^{\prime }dx^{\prime }.
\end{equation*}%
Due to the nonlinearity in our model equation, we can not use the Laplace
transform in the homogenization process. Also our method works even in the
non-Hilbertian framework. Indeed, considering two sequences $(u_{\varepsilon
})_{\varepsilon }$ and $(v_{\varepsilon })_{\varepsilon }$ in $L^{1}(\mathbb{%
R}^{N})$ and $L^{p}(Q)$ respectively satisfying $u_{\varepsilon }\rightarrow
u_{0}$ in $L^{1}(\mathbb{R}^{N})$-strong $\Sigma $ and $v_{\varepsilon
}\rightarrow v_{0}$ in $L^{p}(Q)$-weak $\Sigma $ as $\varepsilon \rightarrow
0$ (where $Q$ is an open subset of $\mathbb{R}^{N}$) we get that $%
u_{\varepsilon }\ast v_{\varepsilon }\rightarrow u_{0}\ast \ast v_{0}$ in $%
L^{p}(Q)$-weak $\Sigma $ as $\varepsilon \rightarrow 0$, where $u_{0}\ast
\ast v_{0}$ is a double convolution with respect to both macroscopic and
microscopic\ variables; see Theorem \ref{t2.2}. The above result was first
proved by Visintin \cite{Visintin2}\ in the periodic setting by using the
two-scale transform or unfolding method. Theorem \ref{t2.2} allows us to
pass to the limit in the convolution terms without using neither the Fourier
transform, nor the Laplace transform, and hence without restricting
ourselves to the Hilbertian setting as it is the case in \cite{Wellander}.
Taking into account the fact that the brain is not necessarily a periodic
medium (even if it can exhibit some kinds of periodicity), we can therefore
emphasize that our work is a true advance in the neural field community.

The paper is organized as follows. In Section 2 we recall some background
material regarding the concept of sigma-convergence. We also prove two
important results which are of independent interest, a general two-scale
convergence result for translates (Proposition \ref{p2.4}) and a general
two-scale convergence result for convolution products (Theorem \ref{t2.2}).
The method used in deriving these results is based on the notion of algebras
with mean value and the concept of sigma-convergence. In Section 3 we prove
existence of solution to the Wilson-Cowan model and derive the a priori
estimate needed for the main homogenization result which is stated and
proved in Section 4. Finally, Section 5 deals with conclusions and outlook.

\section{$\Sigma $-convergence and convolution}

\subsection{Some properties of algebras with mean value}

Let $A$ be an algebra with mean value on $\mathbb{R}^{N}$ (see \cite{20, 38}%
), that is, $A$ is a closed subalgebra of the $\mathcal{C}$*-algebra of
bounded uniformly continuous complex functions $BUC(\mathbb{R}^{N})$ which
contains the constants, is closed under complex conjugation ($\overline{u}%
\in A$ whenever $u\in A$), is translation invariant ($u(\cdot +a)\in A$ for
any $u\in A$ and each $a\in \mathbb{R}^{N}$) and such that each element
possesses a mean value in the following sense:

\begin{itemize}
\item[(\textit{MV})] For each $u\in A$, the sequence $(u^{\varepsilon
})_{\varepsilon >0}$ (where $u^{\varepsilon }(x)=u(x/\varepsilon )$, $x\in 
\mathbb{R}^{N}$) weakly $\ast $-converges in $L^{\infty }(\mathbb{R}^{N})$
to some constant function $M(u)\in \mathbb{C}$ (the complex field).
\end{itemize}

It is known that $A$ (endowed with the sup norm topology) is a commutative $%
\mathcal{C}$*-algebra with identity. We denote by $\Delta (A)$ the spectrum
of $A$ and by $\mathcal{G}$ the Gelfand transformation on $A$. We recall
that $\Delta (A)$ (a subset of the topological dual $A^{\prime }$ of $A$) is
the set of all nonzero multiplicative linear functionals on $A$, and $%
\mathcal{G}$ is the mapping of $A$ into $\mathcal{C}(\Delta (A))$ such that $%
\mathcal{G}(u)(s)=\left\langle s,u\right\rangle $ ($s\in \Delta (A)$), where 
$\left\langle ,\right\rangle $ denotes the duality pairing between $%
A^{\prime }$ and $A$. We endow $\Delta (A)$ with the relative weak$\ast $
topology on $A^{\prime }$. Then using the well-known theorem of Stone (see
e.g., either \cite{21} or more precisely \cite[Theorem IV.6.18, p. 274]{16})
one may easily show that the spectrum $\Delta (A)$ is a compact topological
space, and the Gelfand transformation $\mathcal{G}$ is an isometric
isomorphism identifying $A$ with $\mathcal{C}(\Delta (A))$ (the continuous
functions on $\Delta (A)$) as $\mathcal{C}$*-algebras. Next, since each
element of $A$ possesses a mean value, this induces a mapping $u\mapsto M(u)$
(denoted by $M$ and called the mean value) which is a nonnegative continuous
linear functional on $A$ with $M(1)=1$, and so provides us with a linear
nonnegative functional $\psi \mapsto M_{1}(\psi )=M(\mathcal{G}^{-1}(\psi ))$
defined on $\mathcal{C}(\Delta (A))=\mathcal{G}(A)$, which is clearly
bounded. Therefore, by the Riesz-Markov theorem, $M_{1}(\psi )$ is
representable by integration with respect to some Radon measure $\beta $ (of
total mass $1$) in $\Delta (A)$, called the $M$\textit{-measure} for $A$ 
\cite{26}. It is evident that we have 
\begin{equation}
M(u)=\int_{\Delta (A)}\mathcal{G}(u)d\beta \text{\ for }u\in A\text{.}
\label{2.1}
\end{equation}

The following result is worth recalling. Its proof can be found in \cite%
{DPDE}, and we recall it here for further purposes.

\begin{theorem}
\label{t2.1}Let $A$ be an algebra with mean value on $\mathbb{R}^{N}$. The
translations $T(y):\mathbb{R}^{N}\rightarrow \mathbb{R}^{N}$, $T(y)x=x+y$,
extend to a group of homeomorphisms $T(y):\Delta (A)\rightarrow \Delta (A)$, 
$y\in \mathbb{R}^{N}$, which forms a continuous $N$-dimensional dynamical
system on $\Delta (A)$ whose invariant probability measure is precisely the $%
M$-measure $\beta $ for $A$.
\end{theorem}

\begin{proof}
As $A$ is translation invariant, each translation $T(y)$ induces an
isometric isomorphism still denoted by $T(y)$, from $A$ onto $A$, defined by 
$T(y)u=u(\cdot +y)$ for $u\in A$. Define $\widetilde{T}(y):\mathcal{C}%
(\Delta (A))\rightarrow \mathcal{C}(\Delta (A))$ by 
\begin{equation*}
\widetilde{T}(y)\mathcal{G}(u)=\mathcal{G}(T(y)u)\;\;(u\in A)
\end{equation*}%
where $\mathcal{G}$ denotes the Gelfand transformation on $A$. Then $%
\widetilde{T}(y)$ is an isometric isomorphism of $\mathcal{C}(\Delta (A))$
onto itself; this is easily seen by the fact that $\mathcal{G}$ is an
isometric isomorphism of $A$ onto $\mathcal{C}(\Delta (A))$. Therefore, by
the classical Banach-Stone theorem there exists a unique homeomorphism $%
\overline{T}(y)$ of $\Delta (A)$ onto itself. The family thus constructed is
in fact a continuous $N$-dimensional dynamical system. Indeed the group
property easily comes from the equality $\mathcal{G}(T(y)u)(s)=\mathcal{G}%
(u)(\overline{T}(y)s)$ ($y\in \mathbb{R}^{N}$, $s\in \Delta (A)$, $u\in A$).
As far as the continuity property is concerned, let $(y_{n})_{n}$ be a
sequence in $\mathbb{R}^{N}$ and $(s_{d})_{d}$ be a net in $\Delta (A)$ such
that $y_{n}\rightarrow y$ in $\mathbb{R}^{N}$ and $s_{d}\rightarrow s$ in $%
\Delta (A)$. Then the uniform continuity of $u\in A$ leads to $%
T(y_{n})u\rightarrow T(y)u$ in $BUC(\mathbb{R}^{N})$, and the continuity of $%
\mathcal{G}$ gives $\mathcal{G}(T(y_{n})u)\rightarrow \mathcal{G}(T(y)u)$,
the last convergence result being uniform in $\mathcal{C}(\Delta (A))$.
Hence $\mathcal{G}(T(y_{n})u)(s_{d})\rightarrow \mathcal{G}(T(y)u)(s)$,
which is equivalent to $\mathcal{G}(u)(\overline{T}(y_{n})s_{d})\rightarrow 
\mathcal{G}(u)(\overline{T}(y)s)$. As $\mathcal{C}(\Delta (A))$ separates
the points of $\Delta (A)$, this yields $\overline{T}(y_{n})s_{d}\rightarrow 
\overline{T}(y)s$ in $\Delta (A)$, which implies that the mapping $%
(y,s)\mapsto \overline{T}(y)s$, from $\mathbb{R}^{N}\times \Delta (A)$ to $%
\Delta (A)$, is continuous. It remains to check that $\beta $ is the
invariant measure for $\overline{T}$. But this easily comes from the
invariance under translations' property of the mean value and of the
integral representation (\ref{2.1}). We keep using the notation $T(y)$ for $%
\overline{T}(y)$, and the proof is complete.
\end{proof}

Next, let $B_{A}^{p}$ ($1\leq p<\infty $) denote the Besicovitch space
associated to $A$, that is the closure of $A$ with respect to the
Besicovitch seminorm 
\begin{equation*}
\left\Vert u\right\Vert _{p}=\left( \underset{r\rightarrow +\infty }{\lim
\sup }\frac{1}{\left\vert B_{r}\right\vert }\int_{B_{r}}\left\vert
u(y)\right\vert ^{p}dy\right) ^{1/p}
\end{equation*}%
where $B_{r}$ is the open ball of $\mathbb{R}^{N}$ of radius $r$. It is
known that $B_{A}^{p}$ is a complete seminormed vector space. Moreover, we
have $B_{A}^{q}\subset B_{A}^{p}$ for $1\leq p\leq q<\infty $. The following
properties are worth noticing \cite{27, NA}:

\begin{itemize}
\item[(\textbf{1)}] The Gelfand transformation $\mathcal{G}:A\rightarrow 
\mathcal{C}(\Delta (A))$ extends by continuity to a unique continuous linear
mapping, still denoted by $\mathcal{G}$, of $B_{A}^{p}$ into $L^{p}(\Delta
(A))$, which in turn induces an isometric isomorphism $\mathcal{G}_{1}$, of $%
\mathcal{B}_{A}^{p}=B_{A}^{p}/\mathcal{N}$ onto $L^{p}(\Delta (A))$ (where $%
\mathcal{N}=\{u\in B_{A}^{p}:\mathcal{G}(u)=0\}$). Furthermore if $u\in
B_{A}^{p}\cap L^{\infty }(\mathbb{R}^{N})$ then $\mathcal{G}(u)\in L^{\infty
}(\Delta (A))$ and $\left\Vert \mathcal{G}(u)\right\Vert _{L^{\infty
}(\Delta (A))}\leq \left\Vert u\right\Vert _{L^{\infty }(\mathbb{R}^{N})}$.

\item[(\textbf{2)}] The mean value $M$ viewed as defined on $A$, extends by
continuity to a positive continuous linear form (still denoted by $M$) on $%
B_{A}^{p}$ satisfying $M(u)=\int_{\Delta (A)}\mathcal{G}(u)d\beta $ ($u\in
B_{A}^{p}$). Furthermore, $M(u(\cdot +a))=M(u)$ for each $u\in B_{A}^{p}$
and all $a\in \mathbb{R}^{N}$.

\item[(\textbf{3)}] The dynamical system $T(y)$ is ergodic if and only if
for every $u\in B_{A}^{p}$\ such that $\left\Vert u-u(\cdot +y)\right\Vert
_{p}=0$\ for every $y\in \mathbb{R}^{N}$\ we have $\left\Vert
u-M(u)\right\Vert _{p}=0$.
\end{itemize}

In order to simplify the text, we will henceforth use the same letter $u$
(if there is no danger of confusion) to denote the equivalence class of an
element $u\in B_{A}^{p}$. The symbol $\varrho $ will denote the canonical
mapping of $B_{A}^{p}$ onto $\mathcal{B}_{A}^{p}=B_{A}^{p}/\mathcal{N}$. For 
$u\in \mathcal{B}_{A}^{p}$ (resp. $u\in B_{A}^{p}$) we shall set $\widehat{u}%
=\mathcal{G}_{1}(u)$ (resp. $\widehat{u}=\mathcal{G}(u)$).

In our study we will deal with a special class $\mathbb{A}$ of algebras with
mean value. An algebra with mean value $A$ is in $\mathbb{A}$ if it is
separable (hence its spectrum $\Delta (A)$ is a compact metric space) and
further $\Delta (A)$ has group property. Here below are some examples of
algebras with mean value verifying the above property. 1) Let $A=\mathcal{C}%
_{\text{per}}(Y)$ ($Y=(0,1)^{N}$) denote the algebra of $Y$-periodic complex
continuous functions on $\mathbb{R}^{N}$. It is well-known that $\Delta (A)=%
\mathbb{R}^{N}/\mathbb{Z}^{N}$ (the $N$-torus) which is a separable compact
metrizable topological group. 2) If we denote by $AP(\mathbb{R}^{N})$ the
space of complex continuous almost periodic functions on $\mathbb{R}^{N}$ 
\cite{7, 8}, then for any countable subgroup $\mathcal{R}$ of $\mathbb{R}%
^{N} $, denoting by $AP_{\mathcal{R}}(\mathbb{R}^{N})$ the subspace of $AP(%
\mathbb{R}^{N})$ consisting of functions that are uniformly approximated by
finite linear combinations of functions in the set $\{\gamma _{k}:k\in 
\mathcal{R}\}$ (where $\gamma _{k}(y)=\exp (2i\pi k\cdot y)$, $y\in \mathbb{R%
}^{N}$), we get that $AP_{\mathcal{R}}(\mathbb{R}^{N})\in \mathbb{A}$ with
the further property that its spectrum is a compact topological group
homeomorphic to the dual group $\widehat{\mathcal{R}}$ of $\mathcal{R}$
consisting of characters $\gamma _{k}$ ($k\in \mathcal{R}$) of $\mathbb{R}%
^{N}$; see \cite{AlmostPer} for details. 3) Finally, let $\mathcal{B}%
_{\infty }(\mathbb{R}^{N})$ denote the space of continuous complex functions
on $\mathbb{R}^{N}$ that have finite limit at infinity. It can be easily
shown that its spectrum consists of only one point, the Dirac mass at
infinity $\delta _{\infty }$, hence is a compact topological group. Also $%
\mathcal{B}_{\infty }(\mathbb{R}^{N})$ is separable (see \cite{26}), so that 
$\mathcal{B}_{\infty }(\mathbb{R}^{N})\in \mathbb{A}$. Some other examples
can be considered by taking a combination of the previous ones.

Some notations are in order. Let $A\in \mathbb{A}$. Since $\Delta (A)$ is a
topological group (which we henceforth denote additively), the mapping $%
(s,r)\mapsto s+r$ is continuous from $\Delta (A)\times \Delta (A)$ into $%
\Delta (A)$. $-s$ shall stand for the symmetrization of $s\in \Delta (A)$.
Now, in the above notation, if we take a look at the proof of Theorem \ref%
{t2.1} we observe that the dynamical system associated to the translations
is defined by 
\begin{equation*}
T(y)s=\delta _{y}+s\text{ for }(y,s)\in \mathbb{R}^{N}\times \Delta (A)
\end{equation*}%
where $\delta _{y}$ is the Dirac mass at $y$. With this in mind, for $%
s=\delta _{x}$ and $r=\delta _{y}$ we may see that we have in the same
notations, $s+r=\delta _{x+y}$ and $s-r=\delta _{x-y}$.

We get the following result.

\begin{proposition}
\label{p2.1}Let $A\in \mathbb{A}$. Let $u\in A$ and $a\in \mathbb{R}^{N}$.
Then 
\begin{equation}
\mathcal{G}(u(\cdot +a))=\mathcal{G}(u)(\cdot +\delta _{a}).  \label{2.2}
\end{equation}
\end{proposition}

\begin{proof}
Let us recall that (\ref{2.2}) is equivalent to $\mathcal{G}(u(\cdot +a))(s)=%
\mathcal{G}(u)(s+\delta _{a})$ for any $s\in \Delta (A)$. So, let $y\in 
\mathbb{R}^{N}$, then 
\begin{eqnarray*}
\mathcal{G}(u(\cdot +a))(\delta _{y}) &=&\left\langle \delta _{y},u(\cdot
+a)\right\rangle =u(y+a) \\
&=&\left\langle \delta _{y+a},u\right\rangle =\mathcal{G}(u)(\delta _{y+a})
\\
&=&\mathcal{G}(u)(\delta _{y}+\delta _{a}).
\end{eqnarray*}%
Now, let $s\in \Delta (A)$; there exists a sequence $(y_{n})_{n}\subset 
\mathbb{R}^{N}$ such that $\delta _{y_{n}}\rightarrow s$ in $\Delta (A)$ as $%
n\rightarrow \infty $ (indeed the set $\{\delta _{y}:y\in \mathbb{R}^{N}\}$
is dense in $\Delta (A)$ which is a metric space). Since $\Delta (A)$ is a
topological group, the mapping $(s,r)\mapsto s+r$ is continuous from $\Delta
(A)\times \Delta (A)$ into $\Delta (A)$. So for fixed $r=\delta _{a}$ we
have that $\delta _{y_{n}}+\delta _{a}\rightarrow s+\delta _{a}$, and by the
continuity of $\mathcal{G}(u)$, it comes that $\mathcal{G}(u)(\delta
_{y_{n}}+\delta _{a})\rightarrow \mathcal{G}(u)(s+\delta _{a})$ and $%
\mathcal{G}(u(\cdot +a))(\delta _{y_{n}})\rightarrow \mathcal{G}(u(\cdot
+a))(s)$ as $n\rightarrow \infty $. The uniqueness of the limit gives by the
fact that $\mathcal{G}(u(\cdot +a))(\delta _{y_{n}})=\mathcal{G}(u)(\delta
_{y_{n}}+\delta _{a})$, $\mathcal{G}(u(\cdot +a))(s)=\mathcal{G}(u)(s+\delta
_{a})$.
\end{proof}

\subsection{$\Sigma $-convergence and convolution results}

We begin with the definition of the concept of $\Sigma $-convergence. Let $%
A\in \mathbb{A}$, and let $Q$ be an open subset of $\mathbb{R}^{N}$.
Throughout the paper the letter $E$ will denote any ordinary sequence $%
E=(\varepsilon _{n})$ (integers $n\geq 0$) with $0<\varepsilon _{n}\leq 1$
and $\varepsilon _{n}\rightarrow 0$ as $n\rightarrow \infty $.

\begin{definition}
\label{d2.1}\emph{(1) A sequence }$\left( u_{\varepsilon }\right)
_{\varepsilon >0}\subset L^{p}\left( Q\right) $\emph{\ }$(1\leq p<\infty )$%
\emph{\ is said to }weakly $\Sigma $-converge\emph{\ in }$L^{p}\left(
Q\right) $\emph{\ to some }$u_{0}\in L^{p}(Q;\mathcal{B}_{A}^{p})$\emph{\ if
as }$\varepsilon \rightarrow 0$\emph{, }%
\begin{equation}
\int_{Q}u_{\varepsilon }\left( x\right) \psi ^{\varepsilon }\left( x\right)
dx\rightarrow \iint_{Q\times \Delta (A)}\widehat{u}_{0}\left( x,s\right) 
\widehat{\psi }\left( x,s\right) dxd\beta \left( s\right)  \label{2.3}
\end{equation}%
\emph{for all }$\psi \in L^{p^{\prime }}\left( Q;A\right) $\emph{\ }$\left(
1/p^{\prime }=1-1/p\right) $\emph{\ where }$\psi ^{\varepsilon }\left(
x\right) =\psi \left( x,x/\varepsilon \right) $\emph{\ and }$\widehat{\psi }%
\left( x,\cdot \right) =\mathcal{G}(\psi \left( x,\cdot \right) )$\emph{\ }$%
a.e.$\emph{\ in }$x\in Q$\emph{. We denote this by }$u_{\varepsilon
}\rightarrow u_{0}$\emph{\ in }$L^{p}(Q)$\emph{-weak }$\Sigma $\emph{.}

\noindent \emph{(2) A sequence }$(u_{\varepsilon })_{\varepsilon >0}\subset
L^{p}(Q)$\emph{\ }$(1\leq p<\infty )$\emph{\ is said to }strongly $\Sigma $%
-converge\emph{\ in }$L^{p}(Q)$\emph{\ to some }$u_{0}\in L^{p}(Q;\mathcal{B}%
_{A}^{p})$\emph{\ if it is weakly }$\Sigma $\emph{-convergent and further
satisfies the following condition: }%
\begin{equation}
\left\Vert u_{\varepsilon }\right\Vert _{L^{p}(Q)}\rightarrow \left\Vert 
\widehat{u}_{0}\right\Vert _{L^{p}(Q\times \Delta (A))}.  \label{2.4}
\end{equation}%
\emph{We denote this by }$u_{\varepsilon }\rightarrow u_{0}$\emph{\ in }$%
L^{p}(Q)$\emph{-strong }$\Sigma $\emph{.}
\end{definition}

We have the following result whose proof can be found in \cite{27} (see also 
\cite{Casado}).

\begin{proposition}
\label{p2.2}\emph{(i)} Any bounded sequence $(u_{\varepsilon })_{\varepsilon
\in E}$ in $L^{p}(Q)$ (where $1<p<\infty $) admits a subsequence which is
weakly $\Sigma $-convergent in $L^{p}(Q)$.

\noindent \emph{(ii)} Any uniformly integrable sequence $(u_{\varepsilon
})_{\varepsilon \in E}$ in $L^{1}(Q)$ admits a subsequence which is weakly $%
\Sigma $-convergent in $L^{1}(Q)$.
\end{proposition}

We recall that A sequence $(u_{\varepsilon })_{\varepsilon >0}$\ in $%
L^{1}(Q) $\ is said to be \emph{uniformly integrable}\ if it\ is bounded in $%
L^{1}(Q)$\ and\ further satisfies the property that $\sup_{\varepsilon
>0}\int_{X}\left\vert u_{\varepsilon }\right\vert dx\rightarrow 0$\ for any
integrable set $X\subset Q$\ for which $\left\vert X\right\vert \rightarrow
0 $, where $\left\vert X\right\vert $\ denotes the Lebesgue measure of $X$.

Now, fix $t\in \mathbb{R}^{N}$. Then $(\delta _{t/\varepsilon
})_{\varepsilon >0}$ is a sequence in the compact metric space $\Delta (A)$,
hence it possesses a convergent subsequence still denoted by $(\delta
_{t/\varepsilon })_{\varepsilon >0}$. In the sequel we shall consider such a
subsequence. Let $r\in \Delta (A)$ be such that 
\begin{equation}
\delta _{\frac{t}{\varepsilon }}\rightarrow r\text{ in }\Delta (A)\text{ as }%
\varepsilon \rightarrow 0\text{.}  \label{2.5}
\end{equation}%
Finally let $Q$ be an open subset in $\mathbb{R}^{N}$, and let $%
(u_{\varepsilon })_{\varepsilon >0}$ be a sequence in $L^{p}(Q)$ ($1\leq
p<\infty $) which is weakly $\Sigma $-convergent to $u_{0}\in L^{p}(Q;%
\mathcal{B}_{A}^{p})$. Define the sequence $(v_{\varepsilon })_{\varepsilon
>0}$ as follows: 
\begin{equation*}
v_{\varepsilon }(x)=u_{\varepsilon }(x+t)\text{, }x\in Q-t\text{.}
\end{equation*}%
We then get the following result.

\begin{proposition}
\label{p2.4}Assume \emph{(\ref{2.5})} holds. Then as $\varepsilon
\rightarrow 0$ 
\begin{equation*}
v_{\varepsilon }\rightarrow v_{0}\text{ in }L^{p}(Q-t)\text{-weak }\Sigma
\end{equation*}%
where $v_{0}\in L^{p}(Q-t,\mathcal{B}_{A}^{p})$ is defined by $\widehat{v}%
_{0}(x,s)=\widehat{u}_{0}(x+t,s+r)$ for $(x,s)\in (Q-t)\times \Delta (A)$.
\end{proposition}

\begin{proof}
Let $\varphi \in \mathcal{C}_{0}^{\infty }(Q-t)$ and $\psi \in A$. Let $%
(y_{n})_{n}$ be an ordinary sequence (independent of $\varepsilon $) such
that 
\begin{equation*}
\delta _{y_{n}}\rightarrow r\text{ in }\Delta (A)\text{ as }n\rightarrow
\infty \text{.}
\end{equation*}%
We have 
\begin{eqnarray*}
&&\int_{Q-t}u_{\varepsilon }(x+t)\varphi (x)\psi \left( \frac{x}{\varepsilon 
}\right) dx \\
&=&\int_{Q}u_{\varepsilon }(x)\varphi (x-t)\psi \left( \frac{x}{\varepsilon }%
-\frac{t}{\varepsilon }\right) dx \\
&=&\int_{Q}u_{\varepsilon }(x)\varphi (x-t)\left[ \psi \left( \frac{x}{%
\varepsilon }-\frac{t}{\varepsilon }\right) -\psi \left( \frac{x}{%
\varepsilon }-y_{n}\right) \right] dx \\
&&+\int_{Q}u_{\varepsilon }(x)\varphi (x-t)\psi \left( \frac{x}{\varepsilon }%
-y_{n}\right) dx \\
&=&(I)+(II).
\end{eqnarray*}%
On one hand, as $\varepsilon \rightarrow 0$, 
\begin{equation*}
(II)\rightarrow \iint_{Q\times \Delta (A)}\widehat{u}_{0}(x,s)\varphi (x-t)%
\mathcal{G}(\psi (\cdot -y_{n}))(s)dxd\beta (s).
\end{equation*}%
But we have $\mathcal{G}(\psi (\cdot -y_{n}))=\widehat{\psi }(\cdot -\delta
_{y_{n}})\rightarrow \widehat{\psi }(\cdot -r)$ uniformly in $\mathcal{C}%
(\Delta (A))$ as $n\rightarrow \infty $. Hence 
\begin{eqnarray*}
&&\iint_{Q\times \Delta (A)}\widehat{u}_{0}(x,s)\varphi (x-t)\mathcal{G}%
(\psi (\cdot -y_{n}))(s)dxd\beta (s) \\
&\rightarrow &\iint_{Q\times \Delta (A)}\widehat{u}_{0}(x,s)\varphi (x-t)%
\widehat{\psi }(s-r)dxd\beta (s)\text{ as }n\rightarrow \infty \text{,}
\end{eqnarray*}%
and 
\begin{equation*}
\iint_{Q\times \Delta (A)}\widehat{u}_{0}(x,s)\varphi (x-t)\widehat{\psi }%
(s-r)dxd\beta =\iint_{(Q-t)\times \Delta (A)}\widehat{u}_{0}(x+t,s+r)\varphi
(x)\widehat{\psi }(s)dxd\beta .
\end{equation*}%
On the other hand, 
\begin{eqnarray*}
\left\vert (I)\right\vert &\leq &\int_{Q}\left\vert u_{\varepsilon
}(x)\right\vert \left\vert \varphi (x-t)\right\vert \left\vert \psi \left( 
\frac{x}{\varepsilon }-\frac{t}{\varepsilon }\right) -\psi \left( \frac{x}{%
\varepsilon }-y_{n}\right) \right\vert dx \\
&\leq &c\sup_{z\in \mathbb{R}^{N}}\left\vert \psi \left( z-\frac{t}{%
\varepsilon }\right) -\psi \left( z-y_{n}\right) \right\vert \\
&=&c\left\Vert \psi \left( \cdot -\frac{t}{\varepsilon }\right) -\psi \left(
\cdot -y_{n}\right) \right\Vert _{\infty } \\
&=&c\left\Vert \mathcal{G}\left( \psi \left( \cdot -\frac{t}{\varepsilon }%
\right) -\psi \left( \cdot -y_{n}\right) \right) \right\Vert _{\infty }\text{
(}\mathcal{G}\text{ is an isometry)} \\
&=&c\left\Vert \widehat{\psi }\left( \cdot -\delta _{\frac{t}{\varepsilon }%
}\right) -\widehat{\psi }\left( \cdot -\delta _{y_{n}}\right) \right\Vert
_{\infty }.
\end{eqnarray*}%
Now, using the uniform continuity of $\widehat{\psi }$, we obtain 
\begin{equation*}
\left\Vert \widehat{\psi }\left( \cdot -\delta _{\frac{t}{\varepsilon }%
}\right) -\widehat{\psi }\left( \cdot -\delta _{y_{n}}\right) \right\Vert
_{\infty }\rightarrow 0\text{ as }\varepsilon \rightarrow 0\text{ and next }%
n\rightarrow \infty \text{.}
\end{equation*}%
It therefore follows that, as $\varepsilon \rightarrow 0$, 
\begin{equation*}
\int_{Q-t}u_{\varepsilon }(x+t)\varphi (x)\psi \left( \frac{x}{\varepsilon }%
\right) dx\rightarrow \iint_{(Q-t)\times \Delta (A)}\widehat{u}%
_{0}(x+t,s+r)\varphi (x)\widehat{\psi }(s)dxd\beta .
\end{equation*}%
This concludes the proof.
\end{proof}

The next important result deals with the convergence of convolution
sequences. Let $p\geq 1$ be a real number, and let $(u_{\varepsilon
})_{\varepsilon >0}\subset L^{p}(Q)$ (where we assume here $Q$ to be
bounded) and $(v_{\varepsilon })_{\varepsilon >0}\subset L^{1}(\mathbb{R}%
^{N})$ be two sequences. One may view $u_{\varepsilon }$ as defined in the
whole $\mathbb{R}^{N}$ by taking its zero extension off $Q$. Define 
\begin{equation*}
(u_{\varepsilon }\ast v_{\varepsilon })(x)=\int_{\mathbb{R}%
^{N}}u_{\varepsilon }(t)v_{\varepsilon }(x-t)dt\ \ (x\in \mathbb{R}^{N}).
\end{equation*}%
For $u\in L^{p}(\mathbb{R}^{N};\mathcal{B}_{A}^{p})$ and $v\in L^{1}(\mathbb{%
R}^{N};\mathcal{B}_{A}^{1})$ we define the convolution product $u\ast \ast v$
as follows 
\begin{equation*}
\mathcal{G}_{1}(u\ast \ast v)(x,s):=\iint_{\mathbb{R}^{N}\times \Delta (A)}%
\widehat{u}(t,r)\widehat{v}(x-t,s-r)dtd\beta (r)
\end{equation*}%
for $(x,s)\in \mathbb{R}^{N}\times \Delta (A)$. We denote $\mathcal{G}%
_{1}(u\ast \ast v)$ by $(\widehat{u}\ast \ast \widehat{v})$. This gives a
function in the space $L^{p}(\mathbb{R}^{N};\mathcal{B}_{A}^{p})$ with the
property 
\begin{equation*}
\left\Vert u\ast \ast v\right\Vert _{L^{p}(\mathbb{R}^{N};\mathcal{B}%
_{A}^{p})}\leq \left\Vert u\right\Vert _{L^{p}(\mathbb{R}^{N};\mathcal{B}%
_{A}^{p})}\left\Vert v\right\Vert _{L^{1}(\mathbb{R}^{N};\mathcal{B}%
_{A}^{1})}.
\end{equation*}%
The above inequality can be checked exactly as the Young inequality for
convolution. We have the following result.

\begin{theorem}
\label{t2.2}Let $(u_{\varepsilon })_{\varepsilon }$ and $(v_{\varepsilon
})_{\varepsilon }$ be as above. Assume that, as $\varepsilon \rightarrow 0$, 
$u_{\varepsilon }\rightarrow u_{0}$ in $L^{p}(Q)$-weak $\Sigma $ and $%
v_{\varepsilon }\rightarrow v_{0}$ in $L^{1}(\mathbb{R}^{N})$-strong $\Sigma 
$, where $u_{0}$ and $v_{0}$ are in $L^{p}(Q;\mathcal{B}_{A}^{p})$ and $%
L^{1}(\mathbb{R}^{N};\mathcal{B}_{A}^{1})$ respectively. Then, as $%
\varepsilon \rightarrow 0$, 
\begin{equation*}
u_{\varepsilon }\ast v_{\varepsilon }\rightarrow u_{0}\ast \ast v_{0}\text{
in }L^{p}(Q)\text{-weak }\Sigma \text{.}
\end{equation*}
\end{theorem}

\begin{proof}
Let $\eta >0$ and let $\psi _{0}\in \mathcal{K}(\mathbb{R}^{N};A)$ (the
space of continuous functions from $\mathbb{R}^{N}$ into $A$ with compact
support) be such that $\left\Vert \widehat{v}_{0}-\widehat{\psi }%
_{0}\right\Vert _{L^{1}(\mathbb{R}^{N}\times \Delta (A))}\leq \frac{\eta }{2}
$. Since $v_{\varepsilon }\rightarrow v_{0}$ in $L^{1}(\mathbb{R}^{N})$%
-strong $\Sigma $ we have that $v_{\varepsilon }-\psi _{0}^{\varepsilon
}\rightarrow v_{0}-\psi _{0}$ in $L^{1}(\mathbb{R}^{N})$-strong $\Sigma $,
hence $\left\Vert v_{\varepsilon }-\psi _{0}^{\varepsilon }\right\Vert
_{L^{1}(\mathbb{R}^{N})}\rightarrow \left\Vert \widehat{v}_{0}-\widehat{\psi 
}_{0}\right\Vert _{L^{1}(\mathbb{R}^{N}\times \Delta (A))}$ as $\varepsilon
\rightarrow 0$. So, there is $\alpha >0$ such that 
\begin{equation}
\left\Vert v_{\varepsilon }-\psi _{0}^{\varepsilon }\right\Vert _{L^{1}(%
\mathbb{R}^{N})}\leq \eta \text{ for }0<\varepsilon \leq \alpha \text{.}
\label{2.6}
\end{equation}%
For $f\in \mathcal{K}(Q;A)$, we have (by still denoting by $u_{\varepsilon }$
the zero extension of $u_{\varepsilon }$ off $Q$)%
\begin{eqnarray*}
\int_{Q}(u_{\varepsilon }\ast v_{\varepsilon })(x)f\left( x,\frac{x}{%
\varepsilon }\right) dx &=&\int_{Q}\left( \int_{\mathbb{R}%
^{N}}u_{\varepsilon }(t)v_{\varepsilon }(x-t)dt\right) f\left( x,\frac{x}{%
\varepsilon }\right) dx \\
&=&\int_{\mathbb{R}^{N}}u_{\varepsilon }(t)\left[ \int_{\mathbb{R}%
^{N}}v_{\varepsilon }(x-t)f\left( x,\frac{x}{\varepsilon }\right) dx\right]
dt \\
&=&\int_{\mathbb{R}^{N}}u_{\varepsilon }(t)\left[ \int_{\mathbb{R}%
^{N}}v_{\varepsilon }(x)f\left( x+t,\frac{x}{\varepsilon }+\frac{t}{%
\varepsilon }\right) dx\right] dt \\
&=&\int_{\mathbb{R}^{N}}u_{\varepsilon }(t)\left[ \int_{\mathbb{R}%
^{N}}(v_{\varepsilon }(x)-\psi _{0}^{\varepsilon }(x))f^{\varepsilon }(x+t)dx%
\right] dt \\
&&+\int_{\mathbb{R}^{N}}u_{\varepsilon }(t)\left( \int_{\mathbb{R}^{N}}\psi
_{0}^{\varepsilon }(x)f^{\varepsilon }(x+t)dx\right) dt \\
&=&(I)+(II).
\end{eqnarray*}%
On one hand one has $(I)=\int_{Q}[u_{\varepsilon }\ast (v_{\varepsilon
}-\psi _{0}^{\varepsilon })](x)f^{\varepsilon }(x)dx$ and 
\begin{eqnarray*}
\left\vert (I)\right\vert &\leq &\left\Vert u_{\varepsilon }\right\Vert
_{L^{p}(Q)}\left\Vert v_{\varepsilon }-\psi _{0}^{\varepsilon }\right\Vert
_{L^{1}(\mathbb{R}^{N})}\left\Vert f^{\varepsilon }\right\Vert
_{L^{p^{\prime }}(Q)} \\
&\leq &c\left\Vert v_{\varepsilon }-\psi _{0}^{\varepsilon }\right\Vert
_{L^{1}(\mathbb{R}^{N})}
\end{eqnarray*}%
where $c$ is a positive constant independent of $\varepsilon $. It follows
that 
\begin{equation}
\left\vert (I)\right\vert \leq c\eta \text{ for }0<\varepsilon \leq \alpha 
\text{.}  \label{2.7}
\end{equation}%
On the other hand, in view of Proposition \ref{p2.4}, we have, as $%
\varepsilon \rightarrow 0$, 
\begin{eqnarray*}
\int_{\mathbb{R}^{N}}\psi _{0}^{\varepsilon }(x)f^{\varepsilon }(x+t)dx
&=&\int_{\mathbb{R}^{N}}\psi _{0}^{\varepsilon }(x-t)f^{\varepsilon }(x)dx \\
&\rightarrow &\iint_{\mathbb{R}^{N}\times \Delta (A)}\widehat{\psi }%
_{0}(x-t,s-r)\widehat{f}(x,s)dxd\beta (s),
\end{eqnarray*}%
where $r=\lim \delta _{t/\varepsilon }$ (for a suitable subsequence of $%
\varepsilon \rightarrow 0$) in $\Delta (A)$. So let $\Phi :\mathbb{R}%
^{N}\times \Delta (A)\rightarrow \mathbb{R}$ be defined by 
\begin{equation*}
\Phi (t,r)=\iint_{\mathbb{R}^{N}\times \Delta (A)}\widehat{\psi }%
_{0}(x-t,s-r)\widehat{f}(x,s)dxd\beta (s),\ (t,r)\in \mathbb{R}^{N}\times
\Delta (A).
\end{equation*}%
Then it can be easily checked that $\Phi \in \mathcal{K}(\mathbb{R}^{N};%
\mathcal{C}(\Delta (A)))$, so that there is a function $\Psi \in \mathcal{K}(%
\mathbb{R}^{N};A)$ with $\Phi =\mathcal{G}\circ \Psi $. We can therefore
define the trace $\Psi ^{\varepsilon }(t)=\Psi (t,t/\varepsilon )$ ($t\in 
\mathbb{R}^{N}$) and get 
\begin{eqnarray*}
\Psi ^{\varepsilon }(t) &=&\left\langle \delta _{\frac{t}{\varepsilon }},%
\mathcal{G}(\Psi (t,\cdot ))\right\rangle \\
&=&\left\langle \delta _{\frac{t}{\varepsilon }},\Phi (t,\cdot )\right\rangle
\\
&=&\Phi \left( t,\delta _{\frac{t}{\varepsilon }}\right) =\iint_{\mathbb{R}%
^{N}\times \Delta (A)}\widehat{\psi }_{0}(x-t,s-\delta _{\frac{t}{%
\varepsilon }})\widehat{f}(x,s)dxd\beta (s).
\end{eqnarray*}

Next, we have 
\begin{eqnarray*}
(II) &=&\int_{\mathbb{R}^{N}}u_{\varepsilon }(t)\left( \int_{\mathbb{R}%
^{N}}\psi _{0}^{\varepsilon }(x)f^{\varepsilon }(x+t)dx-\Psi ^{\varepsilon
}(t)\right) dt+\int_{\mathbb{R}^{N}}u_{\varepsilon }(t)\Psi ^{\varepsilon
}(t)dt \\
&=&(II_{1})+(II_{2}).
\end{eqnarray*}%
As for $(II_{1})$, set 
\begin{equation*}
V_{\varepsilon }(t)=\int_{\mathbb{R}^{N}}\psi _{0}^{\varepsilon
}(x)f^{\varepsilon }(x+t)dx-\Psi ^{\varepsilon }(t)\text{ for a.e. }t\in 
\mathbb{R}^{N}.
\end{equation*}%
We claim that, for a.e. $t$, $V_{\varepsilon }(t)\rightarrow 0$ as $%
\varepsilon \rightarrow 0$ (possibly up to a subsequence). Indeed, due to
Proposition \ref{p2.4}, 
\begin{equation*}
\int_{\mathbb{R}^{N}}\psi _{0}^{\varepsilon }(x)f^{\varepsilon
}(x+t)dx\rightarrow \iint_{\mathbb{R}^{N}\times \Delta (A)}\widehat{\psi }%
_{0}(x-t,s-r)\widehat{f}(x,s)dxd\beta (s)\text{ as }\varepsilon \rightarrow 0
\end{equation*}%
where $r$ is such that $\delta _{t/\varepsilon }\rightarrow r$ in $\Delta
(A) $ for some subsequence of $\varepsilon $. Moreover, since $\Psi
^{\varepsilon }(t)=\Phi ^{\varepsilon }(t,\delta _{t/\varepsilon })$, by the
continuity of $\Phi (t,\cdot )$, we have for the same subsequence as $%
\varepsilon \rightarrow 0$ that 
\begin{equation*}
\Psi ^{\varepsilon }(t)\rightarrow \iint_{\mathbb{R}^{N}\times \Delta (A)}%
\widehat{\psi }_{0}(x-t,s-r)\widehat{f}(x,s)dxd\beta (s).
\end{equation*}%
The above claim is justified. Furthermore 
\begin{equation*}
\left\vert V_{\varepsilon }(t)\right\vert \leq c\text{ for a.e. }t\in Q
\end{equation*}%
where $c$ is a positive constant independent of $t$ and $\varepsilon $. On
the other hand, since $Q$ is bounded, we infer from the weak $\Sigma $%
-convergence of $(u_{\varepsilon })_{\varepsilon }$ that $u_{\varepsilon
}\rightarrow \int_{\Delta (A)}\widehat{u}_{0}(\cdot ,r)d\beta (r)$ in $%
L^{1}(Q)$-weak as $\varepsilon \rightarrow 0$. Therefore, by \cite[Lemma 3.4]%
{BBM} it follows that $(II_{1})\rightarrow 0$ as $\varepsilon \rightarrow 0$.

Regarding $(II_{2})$, using once again the weak $\Sigma $-convergence of $%
(u_{\varepsilon })_{\varepsilon }$, we get 
\begin{equation*}
\int_{Q}u_{\varepsilon }(t)\Psi ^{\varepsilon }(t)dt\rightarrow
\iint_{Q\times \Delta (A)}\widehat{u}_{0}(t,r)\widehat{\Psi }(t,r)dtd\beta ,
\end{equation*}%
and 
\begin{eqnarray*}
&&\iint_{Q\times \Delta (A)}\widehat{u}_{0}(t,r)\widehat{\Psi }(t,r)dtd\beta
\\
&=&\iint_{Q\times \Delta (A)}\widehat{u}_{0}(t,r)\Phi (t,r)dtd\beta \\
&=&\iint_{Q\times \Delta (A)}\left[ \iint_{\mathbb{R}^{N}\times \Delta (A)}%
\widehat{u}_{0}(t,r)\widehat{\psi }_{0}(x-t,s-r)dtd\beta (r)\right] \widehat{%
f}(x,s)dxd\beta (s) \\
&=&\iint_{Q\times \Delta (A)}(\widehat{u}_{0}\ast \ast \widehat{\psi }%
_{0})(x,s)\widehat{f}(x,s)dxd\beta (s).
\end{eqnarray*}%
Thus, there is $0<\alpha _{1}\leq \alpha $ such that 
\begin{equation}
\left\vert \int_{Q}(u_{\varepsilon }\ast \psi _{0}^{\varepsilon
})f^{\varepsilon }dx-\iint_{Q\times \Delta (A)}(\widehat{u}_{0}\ast \ast 
\widehat{\psi }_{0})\widehat{f}dxd\beta \right\vert \leq \frac{\eta }{2}%
\text{ for }0<\varepsilon \leq \alpha _{1}\text{.}  \label{2.8}
\end{equation}%
Now, let $0<\varepsilon \leq \alpha _{1}$ be fixed. From the following
decomposition 
\begin{eqnarray*}
&&\int_{Q}(u_{\varepsilon }\ast v_{\varepsilon })f^{\varepsilon
}dx-\iint_{Q\times \Delta (A)}(\widehat{u}_{0}\ast \ast \widehat{v}_{0})%
\widehat{f}dxd\beta \\
&=&\int_{Q}\left[ u_{\varepsilon }\ast (v_{\varepsilon }-\psi
_{0}^{\varepsilon })\right] f^{\varepsilon }dx+\iint_{Q\times \Delta (A)}%
\left[ \widehat{u}_{0}\ast \ast (\widehat{\psi }_{0}-\widehat{v}_{0})\right] 
\widehat{f}dxd\beta \\
&&+\int_{Q}(u_{\varepsilon }\ast \psi _{0}^{\varepsilon })f^{\varepsilon
}dx-\iint_{Q\times \Delta (A)}(\widehat{u}_{0}\ast \ast \widehat{\psi }_{0})%
\widehat{f}dxd\beta ,
\end{eqnarray*}%
we infer from (\ref{2.6})-(\ref{2.8}) that 
\begin{equation*}
\left\vert \int_{Q}(u_{\varepsilon }\ast v_{\varepsilon })f^{\varepsilon
}dx-\iint_{Q\times \Delta (A)}(\widehat{u}_{0}\ast \ast \widehat{v}_{0})%
\widehat{f}dxd\beta \right\vert \leq C\eta \text{ for }0<\varepsilon \leq
\alpha _{1}.
\end{equation*}%
Here $C$ is a positive constant independent of $\varepsilon $. This
concludes the proof.
\end{proof}

\begin{remark}
\label{r2.1}\emph{In this work, we will deal with sequences of functions }$%
(u_{\varepsilon })_{\varepsilon >0}$\emph{\ in the space }$L^{p}(Q_{T})$%
\emph{\ (where} $Q_{T}=Q\times (0,T)$\emph{), and we will say that such a
sequence is weakly }$\Sigma $\emph{-convergent in }$L^{p}(Q_{T})$\emph{\ if} 
\begin{equation*}
\int_{Q_{T}}u_{\varepsilon }\left( x,t\right) f\left( x,t,\frac{x}{%
\varepsilon }\right) dxdt\rightarrow \iint_{Q_{T}\times \Delta (A)}\widehat{u%
}_{0}\left( x,t,s\right) \widehat{f}\left( x,t,s\right) dxdtd\beta \left(
s\right)
\end{equation*}%
\emph{for all }$f\in L^{p^{\prime }}\left( Q_{T};A\right) $\emph{\ }$\left(
1/p^{\prime }=1-1/p\right) $\emph{\ where }$\widehat{f}\left( x,t,\cdot
\right) =\mathcal{G}(f\left( x,t,\cdot \right) )$\emph{\ a.e. in }$(x,t)\in
Q_{T}$\emph{. With the above definition, it is a very easy exercise to see
that all the previous results of the current subsection are carried over
mutatis mutandis to the present setting.}
\end{remark}

\section{Statement of the problem: existence result and a priori estimate}

We consider the parametrized Wilson-Cowan model \cite{WC73} 
\begin{equation}
\left\{ 
\begin{array}{l}
\frac{\partial u_{\varepsilon }}{\partial t}(x,t)=-u_{\varepsilon
}(x,t)+\int_{\mathbb{R}^{N}}J^{\varepsilon }(x-\xi )f\left( \frac{\xi }{%
\varepsilon },u_{\varepsilon }(\xi ,t)\right) d\xi ,\ x\in \mathbb{R}^{N},t>0
\\ 
u_{\varepsilon }(x,0)=u^{0}(x),\ x\in \mathbb{R}^{N}%
\end{array}%
\right.  \label{3.1}
\end{equation}%
where $u_{\varepsilon }$ denotes the electrical activity level field, $f$
the firing rate function and $J^{\varepsilon }=J^{\varepsilon
}(x)=J(x,x/\varepsilon )$ the connectivity kernel. We assume that $J\in 
\mathcal{K}(\mathbb{R}^{N};A)$ (where $A\in \mathbb{A}$) is nonnegative and
is such that $\int_{\mathbb{R}^{N}}J^{\varepsilon }(x)dx\leq 1$, $f:\mathbb{R%
}_{y}^{N}\times \mathbb{R}_{\mu }\rightarrow \mathbb{R}$ is a nonnegative
Carath\'{e}odory function constrained as follows:

\begin{itemize}
\item[(H1)] For almost all $y\in \mathbb{R}^{N}$, the function $f(y,\cdot
):\lambda \mapsto f(y,\lambda )$ is continuous; for all $\lambda \in \mathbb{%
R}$, the function $f(\cdot ,\lambda ):y\mapsto f(y,\lambda )$ is measurable
and $f(\cdot ,0)$ lies in $L^{1}(\mathbb{R}^{N})\cap L^{2}(\mathbb{R}^{N})$;
there exists a positive constant $k_{1}$ such that 
\begin{equation*}
\left\vert f(y,\mu _{1})-f(y,\mu _{2})\right\vert \leq k_{1}\left\vert \mu
_{1}-\mu _{2}\right\vert \text{ for all }y\in \mathbb{R}^{N}\text{ and all }%
\mu _{1},\mu _{2}\in \mathbb{R}.
\end{equation*}
\end{itemize}

An example of a function $f$ that satisfies hypothesis (H1) is $f(y,\lambda
)=g(y)h(\lambda )$ where $g\in \mathcal{K}(\mathbb{R}^{N})$, $g\geq 0$ and 
\begin{equation*}
h(\lambda )=\frac{1}{1+\exp (-\beta (\lambda -\theta ))}\text{ (}\lambda \in 
\mathbb{R}\text{) where }\beta >0\text{ and }\theta \text{ are given.}
\end{equation*}%
It follows from (H1) that, for any function $u\in L^{2}(\mathbb{R}^{N})$,
the function $x\mapsto f(x/\varepsilon ,u(x))$ denoted below by $%
f^{\varepsilon }(\cdot ,u)$, is well defined from $\mathbb{R}^{N}$ to $L^{2}(%
\mathbb{R}^{N})$. Moreover it is an easy exercise to see that 
\begin{equation}
\left\Vert f^{\varepsilon }(\cdot ,u)\right\Vert _{L^{2}(\mathbb{R}%
^{N})}\leq k_{1}\left\Vert u\right\Vert _{L^{2}(\mathbb{R}^{N})}+c_{1}\text{
for all }0<\varepsilon \leq 1  \label{3.0}
\end{equation}%
where $c_{1}=\left\Vert f(\cdot ,0)\right\Vert _{L^{2}(\mathbb{R}^{N})}$.
The following existence result holds true.

\begin{theorem}
\label{t3.1}Let $0<T<\infty $. Assume that $u^{0}\in L^{1}(\mathbb{R}%
^{N})\cap L^{2}(\mathbb{R}^{N})$. Then for each fixed $\varepsilon >0$,
there exists a unique solution $u_{\varepsilon }\in \mathcal{C}([0,\infty
);L^{1}(\mathbb{R}^{N})\cap L^{2}(\mathbb{R}^{N}))$ to \emph{(\ref{3.1})}
which satisfies the estimate 
\begin{equation}
\sup_{0\leq t\leq T}\left[ \left\Vert u_{\varepsilon }(\cdot ,t)\right\Vert
_{L^{1}(\mathbb{R}^{N})}+\left\Vert u_{\varepsilon }(\cdot ,t)\right\Vert
_{L^{2}(\mathbb{R}^{N})}\right] \leq C\text{ for all }\varepsilon >0
\label{3.2}
\end{equation}%
where $C$ is a positive constant depending only on $u^{0}$ and $T$.
\end{theorem}

\begin{proof}
Inspired by the proof of Theorem 2.1 in \cite{Rossi1} (see also \cite{19})
we define the space 
\begin{equation*}
X=\mathcal{C}([0,\rho ];L^{1}(\mathbb{R}^{N})\cap L^{2}(\mathbb{R}^{N}))
\end{equation*}%
with norm 
\begin{equation*}
\left\Vert u\right\Vert _{X}=\max_{t\in \lbrack 0,\rho ]}\left( \left\Vert
u(\cdot ,t)\right\Vert _{L^{1}(\mathbb{R}^{N})}+\left\Vert u(\cdot
,t)\right\Vert _{L^{2}(\mathbb{R}^{N})}\right) ,
\end{equation*}%
where $\rho >0$ is an arbitrary real number to be suitably chosen in the
sequel. Next, by the assumptions on $J$ we have that the trace $%
J^{\varepsilon }$ is well defined and is an element of $\mathcal{K}(\mathbb{R%
}^{N})$ for every $\varepsilon >0$. Let us now consider the operator $%
K_{\varepsilon }:X\rightarrow X$ defined by 
\begin{equation*}
K_{\varepsilon }(\phi )(x,t)=\phi (x,0)+\int_{0}^{t}(J^{\varepsilon }\ast
f^{\varepsilon }(\cdot ,\phi )-\phi )(x,\tau )d\tau .
\end{equation*}%
We observe that a fixed point to $K_{\varepsilon }$ is a local solution $%
u_{\varepsilon }\in X$ to (\ref{3.1}) by the Banach fixed point theorem. We
proceed by showing that $K_{\varepsilon }$ is a strict contraction on $X$.
By the properties of $J^{\varepsilon }$, $f$ and by the use of Young's
inequality, we obtain 
\begin{eqnarray*}
&&\left\Vert K_{\varepsilon }(u)-K_{\varepsilon }(v)\right\Vert _{L^{2}(%
\mathbb{R}^{N})}(t) \\
&\leq &\int_{0}^{\rho }[\left\Vert J^{\varepsilon }\ast (f^{\varepsilon
}(\cdot ,u)-f^{\varepsilon }(\cdot ,v)\right\Vert _{L^{2}(\mathbb{R}%
^{N})}(\tau )+\left\Vert u-v\right\Vert _{L^{2}(\mathbb{R}^{N})}(\tau )]d\tau
\\
&\leq &\int_{0}^{\rho }(k_{1}+1)\left\Vert u-v\right\Vert _{L^{2}(\mathbb{R}%
^{N})}(\tau )d\tau \\
&\leq &(k_{1}+1)\rho \left\Vert u-v\right\Vert _{X}\text{ for every }0\leq
t\leq \rho .
\end{eqnarray*}%
We also have 
\begin{equation*}
\left\Vert K_{\varepsilon }(u)-K_{\varepsilon }(v)\right\Vert _{L^{1}(%
\mathbb{R}^{N})}(t)\leq (k_{1}+1)\rho \left\Vert u-v\right\Vert _{X}\text{
for every }0\leq t\leq \rho .
\end{equation*}%
Choosing $\rho $ small enough such that $2(k_{1}+1)\rho <1$, we obtain that $%
K_{\varepsilon }$ is a strict contraction on $X$. Thus by the Banach fixed
point theorem, there exists a unique local solution $u_{\varepsilon }\in X$
to (\ref{3.1}). Next, arguing exactly as in the proof of Theorem 2.7 in \cite%
{PG}, we get the global existence of the solution to (\ref{3.1}). Now, we
need to check that estimate (\ref{3.2})\ holds true. In order to do that, by
multiplying Eq. (\ref{3.1}) by $u_{\varepsilon }(x,t)$ and integrate the
resulting equality over $\mathbb{R}^{N}$, we obtain 
\begin{equation*}
\frac{\partial }{\partial t}\int_{\mathbb{R}^{N}}\left\vert u_{\varepsilon
}(x,t)\right\vert ^{2}dx=-2\int_{\mathbb{R}^{N}}\left\vert u_{\varepsilon
}(x,t)\right\vert ^{2}dx+2\int_{\mathbb{R}^{N}}u_{\varepsilon
}(x,t)[J^{\varepsilon }\ast f^{\varepsilon }(\cdot ,u_{\varepsilon
})](x,t)dx.
\end{equation*}%
By H\"{o}lder and Young inequalities, we get 
\begin{eqnarray*}
\int_{\mathbb{R}^{N}}u_{\varepsilon }(x,t)[J^{\varepsilon }\ast
f^{\varepsilon }(\cdot ,u_{\varepsilon })](x,t)dx &\leq &\left\Vert
u_{\varepsilon }(\cdot ,t)\right\Vert _{L^{2}(\mathbb{R}^{N})}\left( \int_{%
\mathbb{R}^{N}}\left\vert J^{\varepsilon }\ast f^{\varepsilon }(\cdot
,u_{\varepsilon })\right\vert ^{2}dx\right) ^{\frac{1}{2}} \\
&\leq &\left\Vert u_{\varepsilon }(\cdot ,t)\right\Vert _{L^{2}(\mathbb{R}%
^{N})}\left\Vert J^{\varepsilon }\right\Vert _{L^{1}(\mathbb{R}%
^{N})}\left\Vert f^{\varepsilon }(\cdot ,u_{\varepsilon })\right\Vert
_{L^{2}(\mathbb{R}^{N})}.
\end{eqnarray*}%
We infer from (\ref{3.0}) that 
\begin{equation*}
\left\Vert f^{\varepsilon }(\cdot ,u_{\varepsilon })\right\Vert _{L^{2}(%
\mathbb{R}^{N})}\leq k_{1}\left\Vert u_{\varepsilon }(\cdot ,t)\right\Vert
_{L^{2}(\mathbb{R}^{N})}+c_{1}.
\end{equation*}%
Thus 
\begin{eqnarray*}
\frac{d}{dt}\left\Vert u_{\varepsilon }(\cdot ,t)\right\Vert _{L^{2}(\mathbb{%
R}^{N})}^{2} &\leq &-2\left\Vert u_{\varepsilon }(\cdot ,t)\right\Vert
_{L^{2}(\mathbb{R}^{N})}^{2}+2k_{1}\left\Vert u_{\varepsilon }(\cdot
,t)\right\Vert _{L^{2}(\mathbb{R}^{N})}^{2}+2c_{1}\left\Vert u_{\varepsilon
}(\cdot ,t)\right\Vert _{L^{2}(\mathbb{R}^{N})} \\
&\leq &2k_{1}\left\Vert u_{\varepsilon }(\cdot ,t)\right\Vert _{L^{2}(%
\mathbb{R}^{N})}^{2}+2c_{1}\left\Vert u_{\varepsilon }(\cdot ,t)\right\Vert
_{L^{2}(\mathbb{R}^{N})} \\
&=&2\left\Vert u_{\varepsilon }(\cdot ,t)\right\Vert _{L^{2}(\mathbb{R}%
^{N})}^{2}\left( k_{1}+\frac{c_{1}}{\left\Vert u_{\varepsilon }(\cdot
,t)\right\Vert _{L^{2}(\mathbb{R}^{N})}}\right) .
\end{eqnarray*}%
Either we have 
\begin{equation*}
\left\Vert u_{\varepsilon }(\cdot ,t)\right\Vert _{L^{2}(\mathbb{R}%
^{N})}\leq \frac{2c_{1}}{k_{1}}\text{ for all }\varepsilon >0
\end{equation*}%
or there are some $\varepsilon >0$ for which 
\begin{equation*}
\left\Vert u_{\varepsilon }(\cdot ,t)\right\Vert _{L^{2}(\mathbb{R}^{N})}>%
\frac{2c_{1}}{k_{1}}.
\end{equation*}%
Then, for such $\varepsilon $ we have 
\begin{equation*}
\frac{d}{dt}\left\Vert u_{\varepsilon }(\cdot ,t)\right\Vert _{L^{2}(\mathbb{%
R}^{N})}^{2}\leq 2\left\Vert u_{\varepsilon }(\cdot ,t)\right\Vert _{L^{2}(%
\mathbb{R}^{N})}^{2}\left( k_{1}+\frac{k_{1}}{2}\right) =3k_{1}\left\Vert
u_{\varepsilon }(\cdot ,t)\right\Vert _{L^{2}(\mathbb{R}^{N})}^{2},
\end{equation*}%
hence 
\begin{equation*}
\left\Vert u_{\varepsilon }(\cdot ,t)\right\Vert _{L^{2}(\mathbb{R}%
^{N})}^{2}\leq \exp (3k_{1}t)\left\Vert u^{0}\right\Vert _{L^{2}(\mathbb{R}%
^{N})}^{2}\leq \exp (3k_{1}T\left\Vert u^{0}\right\Vert _{L^{2}(\mathbb{R}%
^{N})}^{2}.
\end{equation*}%
In both cases we see that there exists a positive constant $C$ depending
only on both $u^{0}$ and $T$ such that 
\begin{equation*}
\left\Vert u_{\varepsilon }(\cdot ,t)\right\Vert _{L^{2}(\mathbb{R}%
^{N})}\leq C\text{ for all }0<\varepsilon \leq 1\text{ and all }0\leq t\leq T%
\text{.}
\end{equation*}%
It is also an easy task to see that 
\begin{equation*}
\left\Vert u_{\varepsilon }(\cdot ,t)\right\Vert _{L^{1}(\mathbb{R}%
^{N})}\leq C\text{ for all }\varepsilon >0\text{ and all }0\leq t\leq T\text{%
.}
\end{equation*}%
This completes the proof.
\end{proof}

\begin{remark}
\label{r3.1}\emph{From the uniform boundedness of }$(u_{\varepsilon
})_{0<\varepsilon \leq 1}$\emph{\ in }$\mathcal{C}([0,T];L^{1}(\mathbb{R}%
^{N})\cap L^{2}(\mathbb{R}^{N}))$\emph{, we deduce that }$(u_{\varepsilon
})_{0<\varepsilon \leq 1}$\emph{\ is uniformly integrable in }$L^{1}(\mathbb{%
R}^{N}\times (0,T))$\emph{. Indeed, let }$B\subset \mathbb{R}^{N}\times
(0,T) $\emph{\ be an integrable subset. Denoting by }$\left\vert
B\right\vert $\emph{\ its Lebesgue measure, we have by H\"{o}lder's
inequality that} 
\begin{eqnarray*}
\int_{B}\left\vert u_{\varepsilon }\right\vert dxdt &\leq &\left\vert
B\right\vert ^{\frac{1}{2}}\left\Vert u_{\varepsilon }\right\Vert _{L^{2}(%
\mathbb{R}^{N}\times (0,T))} \\
&\leq &C\left\vert B\right\vert ^{\frac{1}{2}},
\end{eqnarray*}%
\emph{hence }$\sup_{0<\varepsilon \leq 1}\int_{B}\left\vert u_{\varepsilon
}\right\vert dxdt\rightarrow 0$\emph{\ when }$\left\vert B\right\vert
\rightarrow 0$\emph{. We may therefore use [part (ii) of] Proposition \ref%
{p2.2} to deduce the existence of a subsequence of }$(u_{\varepsilon
})_{0<\varepsilon \leq 1}$\emph{\ that weakly }$\Sigma $\emph{-converges in }%
$L^{1}(\mathbb{R}^{N}\times (0,T))$\emph{.}
\end{remark}

\section{Homogenization result}

Let $A$ be an algebra with mean value taken in the class $\mathbb{A}$. In
order to perform the homogenization process, we assume that the function $f$
satisfies the following hypotheses (in which we set $\mathbb{R}_{T}^{N}=%
\mathbb{R}^{N}\times (0,T)$):

\begin{itemize}
\item[(H2)] $f(\cdot ,\mu )\in A$ for all $\mu \in \mathbb{R}$

\item[(H3)] For any sequence $(v_{\varepsilon })_{\varepsilon >0}\subset
L^{1}(\mathbb{R}_{T}^{N})$ such that $v_{\varepsilon }\rightarrow v_{0}$ in $%
L^{1}(\mathbb{R}_{T}^{N})$-weak $\Sigma $, we have $f^{\varepsilon }(\cdot
,v_{\varepsilon })\rightarrow f(\cdot ,v_{0})$ in $L^{1}(\mathbb{R}_{T}^{N})$%
-weak $\Sigma $.
\end{itemize}

The hypothesis (H3) is meaningful. Indeed the convergence result $%
v_{\varepsilon }\rightarrow v_{0}$ in $L^{1}(\mathbb{R}_{T}^{N})$-weak $%
\Sigma $ does not entail the convergence result $f^{\varepsilon }(\cdot
,v_{\varepsilon })\rightarrow f(\cdot ,v_{0})$ in $L^{1}(\mathbb{R}_{T}^{N})$%
-weak $\Sigma $ in general. However there are many situations in which
assumption (H3) is satisfied. Here below are some examples.

\begin{itemize}
\item[1)] Assume that $A=\mathcal{C}_{\text{per}}(Y)$ and that $f(\cdot ,\mu
)$ is $Y$-periodic, $f(y,\cdot )$ is convex. Assume further that $f$
satisfies:

\begin{itemize}
\item[(H4)] $\int_{\mathbb{R}_{T}^{N}}f\left( \frac{x}{\varepsilon }%
,v_{\varepsilon }(x,t)\right) dxdt\rightarrow \iint_{\mathbb{R}%
_{T}^{N}\times Y}f(y,v_{0}(x,t,y))dxdtdy$ as $\varepsilon \rightarrow 0$,
whenever $v_{\varepsilon }\rightarrow v_{0}$ in $L^{1}(\mathbb{R}_{T}^{N})$%
-weak $\Sigma $.
\end{itemize}

\noindent Then $f^{\varepsilon }(\cdot ,v_{\varepsilon })\rightarrow f(\cdot
,v_{0})$ in $L^{1}(\mathbb{R}_{T}^{N})$-weak $\Sigma $ as $\varepsilon
\rightarrow 0$; see \cite[Theorem 4.3]{Visintin1}.

\item[2)] Assume hypotheses (H1)-(H2) hold true. Let $(v_{\varepsilon
})_{\varepsilon >0}$ be a bounded sequence in $L^{1}(\mathbb{R}_{T}^{N})\cap
L^{2}(\mathbb{R}_{T}^{N})$ such that $v_{\varepsilon }\rightarrow v_{0}$ in $%
L^{2}(\mathbb{R}_{T}^{N})$-weak $\Sigma $. Then (possibly up to a
subsequence) we have $v_{\varepsilon }\rightarrow v_{0}$ in $L^{1}(\mathbb{R}%
_{T}^{N})$-weak $\Sigma $. In view of assumption (H1), the sequence $%
(f^{\varepsilon }(\cdot ,v_{\varepsilon }))_{\varepsilon >0}$ is bounded in $%
L^{1}(\mathbb{R}_{T}^{N})\cap L^{2}(\mathbb{R}_{T}^{N})$ so that, up to a
subsequence we have $f^{\varepsilon }(\cdot ,v_{\varepsilon })\rightarrow
z_{0}$ in $L^{2}(\mathbb{R}_{T}^{N})$-weak $\Sigma $ (where $z_{0}\in L^{2}(%
\mathbb{R}_{T}^{N};\mathcal{B}_{A}^{2})$). Now, if we further assume that 
\begin{equation*}
\underset{\varepsilon \rightarrow 0}{\lim \inf }\int_{\mathbb{R}%
_{T}^{N}}f^{\varepsilon }(\cdot ,v_{\varepsilon })v_{\varepsilon }dxdt\leq
\iint_{\mathbb{R}_{T}^{N}\times \Delta (A)}\widehat{z}_{0}\widehat{v}%
_{0}dxdtd\beta ,
\end{equation*}%
then $f^{\varepsilon }(\cdot ,v_{\varepsilon })\rightarrow f(\cdot ,v_{0})$
in $L^{1}(\mathbb{R}_{T}^{N})$-weak $\Sigma $; see \cite[Theorem 8]{Deterhom}%
.

\item[3)] In the special case when $f(y,\lambda )=g(y)h(\lambda )$ with $%
g\in A\cap \mathcal{K}(\mathbb{R}^{N})$, $g\geq 0$ and $h(\lambda )=\frac{1}{%
1+\exp (-\beta (\lambda -\theta ))}$, (H3) is still true. Indeed, as
stressed in Appendix A of \cite{CLSSW2011}, we have $h(v_{\varepsilon
})\rightarrow h(v_{0})$ in $L^{1}(\mathbb{R}_{T}^{N})$-weak $\Sigma $, so
that, since $g\in A$, $f^{\varepsilon }(\cdot ,v_{\varepsilon
})=g^{\varepsilon }h(v_{\varepsilon })\rightarrow gh(v_{0})=f(\cdot ,v_{0})$
in $L^{1}(\mathbb{R}_{T}^{N})$-weak $\Sigma $.
\end{itemize}

We can now state and prove the homogenization result.

\begin{theorem}
\label{t4.1}For any fixed $\varepsilon >0$, let $u_{\varepsilon }$ be the
unique solution of \emph{(\ref{3.1})}. Then as $\varepsilon \rightarrow 0$,
we have 
\begin{equation}
u_{\varepsilon }\rightarrow u_{0}\text{ in }L^{1}(\mathbb{R}_{T}^{N})\text{%
-weak }\Sigma  \label{4.1}
\end{equation}%
where $u_{0}\in \mathcal{C}([0,T];L^{1}(\mathbb{R}^{N};\mathcal{B}_{A}^{1}))$
is the unique solution to the following equation 
\begin{equation}
\left\{ 
\begin{array}{l}
\frac{\partial u_{0}}{\partial t}(x,t,y)=-u_{0}(x,t,y)+(J\ast \ast f(\cdot
,u_{0}))(x,t,y),\ (x,t)\in \mathbb{R}_{T}^{N},\ y\in \mathbb{R}^{N} \\ 
u_{0}(x,0,y)=u^{0}(x),\ x\in \mathbb{R}^{N},\ y\in \mathbb{R}^{N}.%
\end{array}%
\right.  \label{4.2}
\end{equation}
\end{theorem}

\begin{proof}
We infer from Remark \ref{r3.1} that the sequence $(u_{\varepsilon
})_{\varepsilon >0}$ is uniformly integrable in $L^{1}(\mathbb{R}_{T}^{N})$.
So, given an ordinary sequence $E$ it follows from [part (ii) of]
Proposition \ref{p2.2} that there exist a subsequence $E^{\prime }$ of $E$
and a function $u_{0}\in L^{1}(\mathbb{R}_{T}^{N};\mathcal{B}_{A}^{1})$ such
that, as $E^{\prime }\ni \varepsilon \rightarrow 0$, we have (\ref{4.1}). It
now remains to check that $u_{0}$ solves Eq. (\ref{4.2}). Indeed it can be
easily shown that in view of the properties of $f$ and $J$, the solution to (%
\ref{4.2}) is unique, so that, by the uniqueness property we have the
convergence result (\ref{4.1}) for any ordinary sequence $E$ (not only up to
a subsequence $E^{\prime }$), and hence for the whole sequence $\varepsilon
\rightarrow 0$.

Now, since $J\in \mathcal{K}(\mathbb{R}^{N};A)\subset \mathcal{K}(\mathbb{R}%
^{N}\times (0,T);A)$, we have that 
\begin{equation*}
J^{\varepsilon }\rightarrow J\text{ in }L^{1}(\mathbb{R}^{N}\times (0,T))%
\text{-strong }\Sigma \text{ as }E^{\prime }\ni \varepsilon \rightarrow 0%
\text{; see e.g. \cite{26}.}
\end{equation*}%
Hypothesis (H3) together with the convergence result (\ref{4.1}) lead us to 
\begin{equation*}
f^{\varepsilon }(\cdot ,u_{\varepsilon })\rightarrow f(\cdot ,u_{0})\text{
in }L^{1}(\mathbb{R}_{T}^{N})\text{-weak }\Sigma \text{ as }E^{\prime }\ni
\varepsilon \rightarrow 0.
\end{equation*}%
It therefore follows from Theorem \ref{t2.2} and Remark \ref{r2.1} that 
\begin{equation*}
J^{\varepsilon }\ast f^{\varepsilon }(\cdot ,u_{\varepsilon })\rightarrow
J\ast \ast f(\cdot ,u_{0})\text{ in }L^{1}(\mathbb{R}_{T}^{N})\text{-weak }%
\Sigma \text{ as }E^{\prime }\ni \varepsilon \rightarrow 0.
\end{equation*}%
Next, Eq. (\ref{3.1}) is equivalent to the following integral equation 
\begin{equation*}
u_{\varepsilon }(x,t)=u^{0}(x)+\int_{0}^{t}\left[ (J^{\varepsilon }\ast
f^{\varepsilon }(\cdot ,u_{\varepsilon }))(x,\tau )-u_{\varepsilon }(x,\tau )%
\right] d\tau .
\end{equation*}%
Hence, letting $E^{\prime }\ni \varepsilon \rightarrow 0$ and using Fubini
and Lebesgue dominated convergence results in the integral term, we end up
with 
\begin{equation*}
u_{0}(x,t,y)=u^{0}(x)+\int_{0}^{t}\left[ (J\ast \ast f(\cdot ,u_{0}))(x,\tau
,y)-u_{0}(x,\tau ,y)\right] d\tau
\end{equation*}%
which is equivalent to (\ref{4.2}). Moreover this shows that $u_{0}$ lies in 
$\mathcal{C}([0,T];L^{1}(\mathbb{R}^{N};\mathcal{B}_{A}^{1}))$ as expected.
This concludes the proof.
\end{proof}

\section{Conclusions and outlook}

In this paper we have proved some important results which are relevant to
the theory of homogenization in connection with convolution sequences (see
e.g. Theorem \ref{t2.2}). It is to be noted that Theorem \ref{t2.2}
generalizes to the case of algebras with mean value, its counterpart proved
by Visintin \cite{Visintin2} in the special context of the algebras of
continuous periodic functions. This result, based on the so-called sigma
convergence concept, has allowed us to efficiently upscale a heterogeneous
Wilson-Cowan type of models for neural fields. The homogenization result
derived is accurate and more likely can not be achieved through other
classical and conventional methods such as the asymptotic expansions or the
time averaging. Another aspect that should be emphasized is that Theorem \ref%
{t2.2} widely opens the door to many other applications in applied science.
In this regard, it may allow to study homogenization problems in connection
with partial differential equations with fractional order derivatives. Such
kinds of homogenization problems have not yet been solved so far.

\begin{acknowledgement}
\emph{The authors would like to thank the anonymous referee for valuable
remarks and suggestions that helped them to significantly improve they work.}
\end{acknowledgement}

\end{document}